\documentclass[11pt]{article}

\usepackage{amsfonts}
\usepackage{amsmath}
\usepackage{amsthm}

\newtheorem{thm}{Theorem}[section]
\newtheorem{prop}[thm]{Proposition}

\newtheorem{cor}[thm]{Corollary}
\newtheorem{rem}[thm]{Remark}

\numberwithin{equation}{section}

\usepackage{color}

\newcommand{\bel}[1]{\begin{equation}\label{#1}}

\newcommand{\be}{\begin{equation}}
\newcommand{\qe}{\end{equation}}
\newcommand{\ba}{\begin{eqnarray}}
\newcommand{\ea}{\end{eqnarray}}
\newcommand{\rf}[1]{(\ref{#1})}

\def\C{\mathbb C}
\def\R{\mathbb R}

\def\M{\mathcal M}
\def\A{\mathcal A}

\def\H{\mathbb H}
\def\Z{\mathbb Z}
\def\Q{\mathbb Q}

\def\h{\mathfrak h}

\def\proof{\noindent{\em Proof.} }
\def\endproof{\vspace{.1in}}

\begin{document}

\title{Universal moduli spaces of Riemann surfaces}

\author{Lizhen Ji, \\ Department of Mathematics\\
University of Michigan, USA\\ \\
 J\"urgen Jost \\ Max Planck Institute for Mathematics in the Sciences\\
Leipzig, Germany}

\maketitle

\begin{abstract}
We construct a moduli space for Riemann surfaces that is universal in the sense that it represents compact Riemann surfaces of any finite genus. This moduli space is stratified according to genus, and it carries a metric and a measure that induce a Riemannian metric and a finite volume measure on each stratum. Applications to the Plateau-Douglas problem for minimal surfaces of varying genus and to the partition function of Bosonic string theory are outlined. The construction starts with a universal moduli space of Abelian varieties. This space carries a structure of an infinite dimensional locally symmetric space which is of interest in its own right. The key to our construction of the universal moduli space then is the Torelli map that assigns to every Riemann surface its Jacobian and its extension to the Satake-Baily-Borel compactifications.  
\end{abstract}



\tableofcontents
\section{Introduction}

For every $g\geq 0$, let $\M_g$ be the moduli space of compact
Riemann surfaces of genus $g$.
When $g=0$, $\M_g$ consists of only one point, which corresponds to the Riemann sphere $S^2=\C \cup \{\infty\}$.
When $g\geq 1$, it is known that $\M_g$ is noncompact since
Riemann surfaces of positive genus can degenerate. 

A basic result about $\M_g$ is that 
that $\M_g$ is a complex orbifold and a quasi-projective
variety.  See \cite{hm} and references there.

For many applications, we need to compactify $\M_g$, in particular, to obtain compactifications which are  
projective varieties over $\C$, or even defined over some specific number fields.

Besides the well-known  Deligne-Mumford compactification
$\overline{\M}_g^{DM}$ by
adding stable Riemann surfaces of  Euler characteristic $2-2g$,
 there is also the Satake-Baily-Borel 
 compactification $\overline{\M}_g^{SBB}$ whose boundary points
 correspond to unions of compact Riemann surfaces of Euler
 characteristic strictly greater than $2-2g$.
 
 $\overline{\M}_g^{SBB}$ is constructed by Satake-Baily-Borel
 compactification of the Siegel modular variety $\A_g$,
 which is also the moduli space of principally polarized abelian varieties of dimension $g$.
 The compatification  $\overline{\M}_g^{SBB}$ will be described in more detail
 in the next section

The two compactifications  $\overline{\M}_g^{DM}$
and  $\overline{\M}_g^{SBB}$ are different in several aspects.
\begin{enumerate}
\item $\overline{\M}_g^{DM}$ is a complex orbifold,
but $\overline{\M}_g^{SBB}$ is highly singular when $g\geq 2$,

\item $\overline{\M}_g^{DM}$  is the moduli space of stable
Riemann surfaces of Euler characteristic $2-2g$ for $g\geq 2$,
but $\overline{\M}_g^{SBB}$ is not a moduli space 
for a modular functor (see \cite[p. 45]{hm}). 

\item The boundary points of $\overline{\M}_g^{DM}$ correspond to
Riemann surfaces with punctures by pinching along  
loops of Riemann surfaces $\Sigma_g$, and the boundary points of $\overline{\M}_g^{SBB}$ correspond to unions of compact 
Riemann surfaces which are obtained from
the punctured Riemann surfaces in the boundary of 
$\overline{\M}_g^{DM}$ by forgetting (or filling in)
 the punctures.
 Therefore, there is a surjective map $\overline{\M}_g^{DM} \to \overline{\M}_g^{SBB}$, which is not injective when $g\geq 2$.
 
\end{enumerate}

For many applications, especially those in algebraic geometry and arithmetic geometry, the above properties,
especially the modular property, make $\overline{\M}_g^{DM}$ more desirable.
On the other hand, for some applications to
string theory and minimal surfaces,
$\overline{\M}_g^{SBB}$ is more suitable for the reason that
only compact Riemann surfaces appear on the boundary
of $\overline{\M}_g^{SBB}$ and only compact Riemann 
surfaces matter, for the basic reason that  the Riemann extension theorem can
remove the punctures.

{Let us explain this in more detail. Let us start with geometric analysis, with the Plateau-Douglas problem in Euclidean space (or some Riemannian manifold). Here, one wants to find conditions for a configuration of $k$ disjoint oriented Jordan curves in Euclidean space to bound a minimal surface of a given genus $g$ and $k$ boundary curves. A minimal surface in this context is a harmonic and conformal map from some  Riemann surface $\Sigma$ satisfying the above Plateau type boundary condition \cite{dhs}.  In order to break the diffeomorphism invariance of the area integral, one works with the Dirichlet integral, that is, for a map $h:C\to \R^N$, we consider
\bel{intro1}
S(h,C):=\int_C |dh|^2; 
\qe
by conformal invariance, we do not need to specify a conformal metric on $\Sigma$. That is, when $\rho(z)dzd\bar{z}$ is a Riemannian metric on $\Sigma$ compatible with its conformal structure, then $\int_C \rho(z)^{-1}\frac{\partial h}{\partial z}\frac{\partial h}{\partial \bar{z}} \rho(z)dzd\bar{z} =S(h,C)$ independently of the particular choice of $\rho$.

When one takes a minimizing sequence for the Dirichlet integral \rf{intro1}, again such a sequence could degenerate and end up with a limit of smaller topological type, that is, an element of the boundary of $\M_g$.}

{Since here we  discuss surfaces with boundary that have to map to  the
given disjoint Jordan curves,  we need  moduli spaces of Riemann surfaces with boundary. These can be obtained from those without boundary as follows (see e.g. \cite{j1} for more details). Let $\Sigma$ be a compact Riemann surface of genus $g$ with $k$ boundary curves. We then form its Schottky double $\Sigma'$, a compact Riemann surface without boundary of genus $2g+k-1$, by reflecting $\Sigma$ across its boundary, that is, by taking a copy $\Sigma''$ of $\Sigma$ with the opposite orientation and identifying $\Sigma$ and $\Sigma''$ along their corresponding boundaries. $\Sigma'$ then possesses an anticonformal involution $i$ that interchanges $\Sigma$ and $\Sigma'$, leaving their common boundaries fixed. (When we equip $\Sigma$ with a constant curvature metric for which all boundary curves are geodesic, we can also perform these constructions within the Riemannian setting. $i$ then is an orientation reversing isometry whose fixed point is a collection of closed geodesics that constitute the common boundary of $\Sigma$ and $\Sigma''$. Conversely, if we have a compact Riemann surface $\Sigma'$ of genus $2g+k-1$ with such an involution that leaves $k$ geodesics fixed, then $\Sigma'$ can be seen as the union of two isometric surfaces of genus $g$ with $k$ boundary curves. From this construction, we see that the moduli space of genus $g$ surfaces with $k$ boundary curves is the moduli space of genus $2g+k-1$ surfaces without boundary with an involution that leaves $k$ disjoint simple loops (closed geodesics with respect to a constant curvature metric) fixed. The moduli space of surfaces with such an involution $\M_{g,k}$ is a totally real subspace of the moduli space $\M_{2g+k-1}$, and all properties of the latter apply to the former with obvious modifications, except that $\M_{g,k}$, as a totally real subspace, does not possess a complex structure. Therefore, in the sequel, we only discuss the spaces $\M_g$. }

{Returning to the  discussion of \rf{intro1} when the underlying Riemann surface $\Sigma$ degenerates in $\M_g$, the question arises which boundary of $\M_g$ should we take here when we want to consider limits of sequences of degenerating Riemann surfaces and assign a value to the Dirichlet integral on such a limit.  The key observation is that  if we take a limit of harmonic maps, we should expect the limit also to be harmonic. That is, we get a harmonic map from a degenerated surface. When we consider that object as an element of the Deligne-Mumford compactification $\overline{\M}_g^{DM}$, it would have two punctures. But a (bounded) harmonic map extends across such a puncture, and therefore, it does not feel the effect of that puncture. Thus, the punctures are irrelevant, and the natural domain for our harmonic map is a lower topological type Riemann surface without any punctures, that is, an element of the  Satake-Baily-Borel $\overline{\M}_g^{SBB}$
 rather than the Deligne-Mumford compactification \cite{js}. Moreover, if one wants to look at minimal surfaces of arbitrary genus, one should have a universal moduli space that contains Riemann surfaces of all posssible genera. In some sense, we want to let $g\to \infty$. But from the Deligne-Mumford compactification we would then encounter boundary strata with ever more punctures, and in the limit infinitely many. This seems undesirable. Thus, this provides motivation for our construction of a universal moduli space. 
 
String theory \cite{pol} can be considered as a quantization of the Plateau-Douglas problem just described \cite{j2}. And when one wants to compute corresponding partition and correlation functions, one should take a sum over all possible genera, possibly with suitable weights for the different values of $g$, of suitable integrals over the individual $M_g$ \cite{dhp}, as we shall explain in more detail in Section \ref{sec:int}. But then again, a boundary stratum of $\M_{g}$ should be $\M_{g-1}$ and or a   product  $\M_{g_1}\times \cdots \times \M_{g_k}$ with $g_1 +
\cdots+g_k=g$, and not a space blown up from the latter by introducing additional puncture positions  as in the Deligne-Mumford
compactification. In any case, the situation for genus $g-1$ should be structurally the same as for genus $g$, and not more complicated. Otherwise, we cannot meaningfully let $g\to \infty$. }

{For these applications, it is therefore important to consider surfaces of different genus together. For instance, for a general treatment of the Plateau-Douglas problem for minimal surfaces, we wish to have a Conley type  index formula that involves minimal surfaces of all finite genera simultaneously. This  is because given some configuration of Jordan curves in Euclidean space or a Riemannian manifold, we may not know a priori what the largest genus of a minimal surface bounded by this configuration is. In particular, we want to have a global Euler characteristic involving all critical points of the Dirichlet integral of any genus. Also, when we vary the configuration of Jordan curves, while the global Euler characteristic will stay the same, the solutions can change their genera or bifurcate in other ways. Likewise, when one wants to do string theory in a non-perturbative manner, one wants to have a formula that includes all finite genera simultaneously, instead of an expansion in terms of the genus.}

The above consideration shows that it is desirable to
construct a {\em universal compactified moduli space $\overline{\M}_\infty$}
satisfying the following properties:
\begin{enumerate}
\item $\overline{\M}_\infty$ is a connected stratified complex analytic space
such that the closure of each stratum is a projective variety.
\item For every $g$, $\overline{\M}_g^{SBB}$ is embedded into
$\overline{\M}_\infty$, and $\overline{\M}_\infty$ is the union of
these subsets $\overline{\M}_g^{SBB}$, $g\geq 0$.
\item {There is a natural measure on $\overline{\M}_\infty$ 
which induces compatible measures on different strata,
and with respect to the induced measure, every stratum has finite volume. 
}
\end{enumerate}

We note that the connectedness condition in (1) excludes  the trivial
construction by taking the
disjoint union of $\overline{\M}_g^{SBB}$, $g\geq 0$. 

The idea is to construct a natural embedding $\overline{\M}_g^{SBB}\subset \overline{\M_{g+1}}^{SBB}$ for every $g$, which are compatible
for all $g$,  then take 
their union under such inclusion obtain a desired space.
This can be obtained by using another  family of locally symmetric
spaces $\A_g$ and their compactifications 
$\overline{\A}_g^{SBB}$, and the Jacobian map
$\M_g\to \A_g$ allows us to pass such a construction to
$\overline{\M}_g^{SBB}$. This  infinite
dimensional locally symmetric space structure will allow us to define a (positive semi-definite) Riemannian metric and 
a natural measure on the universal moduli space $\overline{\M}_\infty$. A slight modification, consisting in pulling back the metric from its Jacobian back to each individual Riemann surface instead of pulling the metric from the moduli space of principally polarized Abelian varieties back to the moduli space of Riemann surfaces, will even lead to a positive definite Riemannian metric.

\section{Satake-Baily-Borel compactification of the moduli space 
 $\M_g$}

 As mentioned before, the Deligne-Mumford compactification
 $\overline{\M_g}^{DM}$ is well-known, while
 the Satake-Baily-Borel compactification $\overline{\M_g}^{SBB}$
 is less known. Hence, we give the definition of the latter in some detail.
 
 Let $$\h_g=\{X+iY\mid X, Y \text{\ symmetric\ } n\times n \text{\ matrices}, Y>0\}$$ be the {\em Siegel upper half space of degree} $g$.
 It is a Hermitian symmetric space of noncompact type.
 The symplectic group $Sp(2g, \R)$ acts holomorphically and transitively 
 on it with the stabilizer of $i I_g$ isomorphic to $U(g)$.
 Therefore, $$\h_g=Sp(2g, \R)/U(g).$$
 The Siegel modular group $Sp(2g, \Z)$ acts properly and holomorphically
 on $\h_g$,
 and the quotient $Sp(2g, \Z)\backslash \h_g$ is a complex orbifold and equal to $\A_g$. See the book \cite{nam} for more detail.
 
 Since Abelian varieties can degenerate, $\A_g$ is noncompact.
 It is a quasi-projective variety and can be compactified to
 a normal projective variety $\overline{\A}_g^{SBB}$.
 This compactification of $\A_g$ was first constructed as a topological
 space by Satake in \cite{sa} and as a normal
 projective space by Baily in \cite{ba}. Since this is a special
 case of the Baily-Borel compactification for general
 arithmetic locally Hermitian symmetric spaces in \cite{bb},
 we call it {\em Satake-Baily-Borel compactification} and denote it by
 $\overline{\A}_g^{SBB}$. 
 
 Since $\h_g$ is a Hermitian symmetric space of noncompact type,
 it can be embedded into its compact dual, under the Borel embedding. (Note that when $g=1$, $\h_1$ is the Poincar\'e upper halfplane $\H^2=\{z\in \C\mid \text{Im} (z)>0\}$, and its compact dual
 is the Riemann sphere $\C\cup \{\infty\}$).
 Denote the closure  of $\h_g$ under this embedding by $\overline{\h_g}$.
 Then the symplectic group $Sp(2g, \R)$ acts on the compactification 
 $\overline{\h_g}$. 
 
 For every $g'< g$, we can embed $\h_{g'}$ into the boundary
 of $\overline{\h_g}$ in infinitely many different ways.
 The most obvious one, usually called the {\em standard embedding} 
 \cite{nam},  is as follows:
 
 \begin{equation}\label{embed-bdy}
 \h_{g'}\hookrightarrow \overline{\h_g}, \quad X'+iY' \mapsto 
 \begin{pmatrix} X' & 0\\ 0 & 0
 \end{pmatrix}
 +i \begin{pmatrix} Y' & 0\\ 0 & 0
 \end{pmatrix}. 
 \end{equation}
 
 Under the action of $Sp(2g, \R)$, we get other embeddings.
 To compactify $\A_g$, we only need 
 the translates of $\h_{g'}$, $g'< g$,
 under $Sp(2g, \Z)$. These boundary components
 are called  {\em rational
 boundary components} of $\h_g$.
 
 Denote the union of $\h_g$ with these rational boundary components by $\overline{\h_g}_{,\Q}$.
 Then there is a Satake topology on $\overline{\h_g}_{,\Q}$
 such that $Sp(2g, \Z)$ acts continuously
 with a compact quotient, which is the compactification
 $\overline{\A}_g^{SBB}$.
 
 Note that the action of $Sp(2g, \Z)$ on $\overline{\h_g}_{,\Q}$
 is not proper, but the quotient $Sp(2g, \Z)\backslash \overline{\h_g}_{,\Q}$ is a Haudorff topology.
 A good example to keep this in mind is to consider the case $g=1$, or equivalently the case of $SL(2, \Z)$ acting on
 the upper halfplane $\H^2$.
 The boundary of the closure $\overline{\H^2}$ of $\H^2$ in the extended complex
 plane $\C \cup \{+\infty\}$ is equal to $\partial \H^2=\R \cup \{\infty\}$. In the boundary $ \partial \H^2$, rational boundary components
 correspond to the rational points $\Q\cup\{\infty\}$.
 On the other hand, the union $\H^2\cup \Q\cup\{\infty\}$
 with the induced subspace topology from $\overline{\H^2}$
 is not the Satake topology. In fact, with respect to the induced
 subspace topology,  its quotient by
 $SL(2, \Z)$ is non-Hausdorff.
 In the Satake topology of $\H^2\cup \Q\cup\{\infty\}$,
 a basis of neighborhoods of every rational boundary
 point $\xi \in \Q\cup \{\infty\}$ is given by horoballs based
 at $\xi$.

 It is known that  for every  compact Riemann
surface $\Sigma_g$ of genus $g\geq 1$,
the complex torus 
$H^0(\Sigma, \Omega)^*/H_1(\Sigma_g, \Z)$
has a canonical principal polarization induced from the
cup product of $H_1(\Sigma, \Z) \times H^1(\Sigma, \Z)\to 
H^2(\Sigma, \Z)\cong \Z$.
It is called the {\em Jacobian variety} of $\Sigma_g$ and denoted
by $J(\Sigma_g)$.

Using this, we can define the  {\em Jacobian} (or {\em period,
Torelli}) map, 
$$J: \M_g \to \A_g, \quad \Sigma_g \mapsto J(\Sigma_g).$$ 

By Torelli theorem, $J$ is injective (see \cite{gh} for example).
By \cite{ba2}, the image $J(\M_g)$ is an algebraic subvariety.
By \cite{os},  $J$ is an embedding between complex algebraic
varieties. 
Therefore, the closure of $J(\M_g)$ in $\overline{\A}_g^{SBB}$
is a projective variety,  and it gives  a compactification of $\M_g$,
which is  the {\em Satake-Baily-Borel compactification} $\overline{\M}_g^{SBB}$ mentioned above.

\section{Universal moduli spaces of abelian varieties}
 
 In this section, we
 construct an infinite dimensional  locally symmetric space
 $\A_\infty$ and its completion $\overline{\A}^{SBB}_\infty$
 such that:
 \begin{enumerate}
 \item For every $g$, there is a canonical inclusion $\A_g\hookrightarrow \A_\infty$.
 \item For every $g$, there is an canonical embedding $\A_g\hookrightarrow \A_{g+1}$,
 and these inclusions are compatible for all $g$.
 \item $\A_\infty$ is the union of these images $\A_g$.
 \item $\overline{\A}^{SBB}_\infty$ is a stratified complex analytic
 space such that the closure of  each stratum is a complex projective space.
 \item $\A_\infty$ is open and dense in $\overline{\A}^{SBB}_\infty$.
 For every $g$, the closure of $\A_g$ in $\overline{\A}^{SBB}_\infty$
 is $\overline{\A}_g^{SBB}$.
 
 \item $\overline{\A}^{SBB}_\infty$ is the union of the subspaces
 $\overline{\A}_g^{SBB}$, $g\geq 0$.
 \end{enumerate}
 
 For this purpose, we first construct  an infinite dimensional symmetric space $\h_\infty$ from the family of Siegel upper halfspaces $\h_g$, $g\geq 1$.
 
 For every $g\geq 1$, we embed $\h_g$ into $\h_{g+1}$ as follows:
 
 \begin{equation}\label{embed-h}
 X+i Y \in \h_g \mapsto \begin{pmatrix} X & 0 \\ 0 & 0 \end{pmatrix}
 + i
 \begin{pmatrix} Y & 0 \\ 0 & 1 \end{pmatrix}
\in \h_{g+1}.
 \end{equation}
 
 Clearly, the image of $\h_g$ in $\h_{g+1}$ is a totally geodesic
 subspace.
 Then we obtain a direct sequence of increasing Hermitian symmetric
 spaces of noncompact type:
 
 $$\h_1 \hookrightarrow \h_2 \hookrightarrow \cdots.$$
 
 Then the direct limit  with the natural topology
 $\h_\infty=\lim_{g\to \infty}\h_g$ is an infinite
 dimensional smooth manifold locally based on $\C^\infty$,
 the complex vector space of finite sequences with the finite
 topology. See \cite{gl} \cite{ha}.
 
 In our case, we can be more specific with the space $\h_\infty$ about
 its nature as an infinite dimensional manifold, which will be important for applications we have in mind.
 
\begin{equation}\label{h-infty}
\h_\infty=\{ \begin{pmatrix} X & 0\\ 0 & 0\end{pmatrix}
+ i \begin{pmatrix} Y & 0\\ 0 & I_\infty\end{pmatrix}
\mid X, Y \text{ are $g\times g$ symmetric matrices for some $g$}, Y>0\}.
\end{equation}

Abstractly, we have the following result.

\begin{prop}
The limit space $\h_\infty$ is an { infinite dimensional Hermitian symmetric space}. 
\end{prop}
\begin{proof}
Since each $\h_g$ is a totally geodesic Hermitian subspace of 
a larger Hermitian symmetric
space $\h_{g+k}$, $k\geq 0$,  it can be seen that  every point in $\h_\infty$ is an isolated fixed
 point of an involutive  holomorphic isometry  of $\h_\infty$,
 and hence  $\h_\infty$ is an  infinite dimensional  symmetric space.  
 \end{proof}

  On the other hand, 
 the description of $\h_\infty$ in Equation \ref{h-infty} shows that 
 it is not a usual infinite
 dimensional Hermitian symmetric space modelled on Hilbert
 manifolds as in \cite{kau} \cite{tu} \cite{up}, or other completions
 and extensions in other papers.
 
 For example, one natural extension of $\h_\infty$ is to consider
 the symmetric space defined by
 $$\h_\infty'=\{X+iY\mid X, Y \text{\ are\ } \infty\times \infty \text{-symmetric matrixes
 whose entries satisfy some}$$
 $$\text{ convergence properties, 
 and  finite major minors of $Y$ are positive definite.} \}$$
 
 This is more common in the theory of infinite dimensional
 symmetric spaces. 
 A more restricted extension of $\h_\infty$ is to consider the 
 symmetric space
 
 $$\h_\infty''=\{X+iY\in \mid X, Y \in\h_\infty',  X, Y-I_\infty
 \text{\ have only finitely many nonzero entries}.\}$$
 From Equation \ref{h-infty}, it is clear that we have strict inclusion:
 $$\h_\infty\subset \h_\infty'' \subset \h_\infty'.$$

 As mentioned before, for each $g$, $Sp(2g, \Z)$ acts properly,
 homolorphically and isometrically 
 on $\h_g$, and the quotient $Sp(2g, \Z)\backslash
 \h_g$  is a Hermitian locally symmetric space and
 is equal to $\A_g$.
 
 We need to construct a discrete group $Sp(\infty, \Z)$ which 
 also acts
 properly, holomorphically and isometrically  on $\h_\infty$ such that the quotient $Sp(\infty, \Z)\backslash \h_\infty$ gives the desired space $\A_\infty$, an infinite dimensional Hermitian 
 locally symmetric space.
 
 For each $g$, every element of $Sp(2g, \R)$ can be written
 as a $2\times 2$ block matrix  consisting of $n\times n$ matrices,
 $$M=\begin{pmatrix} A & B\\ C & D\end{pmatrix},$$
 such that $${}^t M \begin{pmatrix} 0 & I_g\\-I_g & 0\end{pmatrix}
 M=\begin{pmatrix} 0 & I_g\\-I_g & 0\end{pmatrix},$$
 which is equivalent to the equations:
 $${}^t A C= {}^t C A, {}^t B D= {}^t D B, {}^t A D- {}^t C B= I_g.$$
 
 Then we get an embedding
 $$Sp(2g, \R) \to Sp(2(g+1), \R), \quad
 \begin{pmatrix} A & B\\ C & D\end{pmatrix}
 \to \begin{pmatrix} 
 \begin{pmatrix} A & 0\\ 0 & 1\end{pmatrix} 
 & 
 \begin{pmatrix} B& 0\\ 0 & 0\end{pmatrix}\\ 
 \begin{pmatrix} C& 0\\ 0 & 0\end{pmatrix}
 & 
 \begin{pmatrix} D& 0\\ 0 & 1\end{pmatrix}
 \end{pmatrix}.
 $$
 
 This induces an embedding
 
 $$Sp(2g, \Z)\to Sp(2(g+1), \Z).$$
 
 Under these embeddings of $\h_g\hookrightarrow \h_{g+1}$
 and $Sp(2g, \Z)\to Sp(2(g+1), \Z)$, it is clear that
 the action of $Sp(2g, \Z)$ on $\h_{g+1}$
 leaves the subspace $\h_g$ stable,
 and we obtain a canonical embedding
 
 \begin{equation}\label{inclusion-a}
  \A_g=Sp(2g, \Z)\backslash \h_g \to
 \A_{g+1}=Sp(2(g+1), \Z)\backslash \h_{g+1}.
 \end{equation}

 From these increasing sequences of groups $Sp(2g, \R)$ and 
  $Sp(2g, \Z)$, we obtain two
 limit groups
 $$Sp(\infty, \R)=\lim_{g\to \infty} Sp(2g, \R), \quad
  Sp(\infty, \Z)=\lim_{g\to \infty} Sp(2g, \Z).$$

  These groups can be described explicitly as follows:
 $$  Sp(\infty, \R)=\{ 
  \begin{pmatrix} 
  \begin{pmatrix} A & 0\\ 0 & I_\infty \end{pmatrix}

  & 
 \begin{pmatrix} B & 0\\ 0 & 0 \end{pmatrix}
  \\
  \begin{pmatrix} C & 0\\ 0 & 0 \end{pmatrix} 
  & 
  \begin{pmatrix} D & 0\\ 0 & I_\infty \end{pmatrix}  
  \end{pmatrix}
 \}, $$
  where for some $g\geq 1$,
   $A, B, C$ and $D$ are $g\times g$ block matrices,  and they satisfy
 $$   \begin{pmatrix} A & B\\ C & D \end{pmatrix}   
  \in Sp(2g, \R).   $$
  
  Similarly, 
  $$  Sp(\infty, \Z)=\{ \begin{pmatrix} 
  \begin{pmatrix} A & 0\\ 0 & I_\infty \end{pmatrix}

  & 
 \begin{pmatrix} B & 0\\ 0 & 0 \end{pmatrix}
  \\
  \begin{pmatrix} C & 0\\ 0 & 0 \end{pmatrix} 
  & 
  \begin{pmatrix} D & 0\\ 0 & I_\infty \end{pmatrix}  
  \end{pmatrix}
  \}, $$
  where $A, B, C, D$ are $g\times g$-block integral matrices and satisfy
 $$   \begin{pmatrix} A & B\\ C & D \end{pmatrix}   
  \in Sp(2g, \Z).   $$

  It is clear that $Sp(\infty, \R)$ acts transitively on $\h_\infty$,
  and hence $\h_\infty$ can be considered as a symmetric space
  associated with the infinite dimensional Lie group $Sp(\infty, \R)$.
  
  Though $Sp(\infty, \Z)$ is not a finitely generated group, it is countable
  and is a discrete subgroup of the Lie group  $Sp(\infty, \R)$.
  It acts properly on $\h_\infty$,
  and the quotient  space $Sp(\infty, \Z)\backslash \h_\infty$
  is Hausdorff,
  and  it can be seen that 
  
 \begin{equation}\label{a-defi}
 Sp(\infty, \Z)\backslash \h_\infty=\lim_{g\to \infty}
 \A_g,
 \end{equation}
 where the right hand side is defined by the inclusion 
 in Equation  \ref{inclusion-a}.  Denote {\em this limiting space}
  by $\A_\infty$. We call it {\em the universal moduli space of principally
  polarized abelian varieties}. 
  
  By the explicit description of $\h_\infty$ in Equation \ref{h-infty},
  we can see that $\h_\infty$ is a complex manifold and a Hermitian symmetric space, and the action of $Sp(\infty, \Z)$
  on $\h_\infty$ is holomorphic.  
 
 \begin{rem}
 {\em If we take the limit space $\lim_{g\to \infty}
 \A_g$ directly, we only get a topological space, or some kind of
 infinite dimensional orbitfolds, since each $\A_g$ is not smooth.
 The realization as a quotient of $\h_\infty$ gives it more structures, which will be needed for constructing natural measures
 on $\A_\infty$ and its completion $\overline{\A}_\infty^{SBB}$.
 On the other hand, it might be
 possible to generalize the constructions \cite{gl} \cite{ha} to the setup
 of orbifolds.
 In any case, the triple of an infinite dimensional Lie group
 $Sp(\infty, \R)$, an infinite dimensional symmetric space
 $\h_\infty$, and an arithmetic group $Sp(\infty, \Z)$ is appealing.}
 \end{rem}
 
 \begin{rem}
 {\em The group $Sp(\infty, \Z)$ can be thought of as an
 arithmetic subgroup of the infinite dimensional Lie group
 $Sp(\infty, \R)$. 
 Though infinite dimensional symmetric spaces
 have been studied by many people, see for example, 
 \cite{kau} \cite{tu} \cite{up}, it seems that their quotients 
 by analogues of arithmetic subgroups of linear 
 algebraic groups have not been studied.
 On the other hand, it is known that arithmetic locally symmetric spaces have much richer structures than
 symmetric spaces of noncompact types, and they occur naturally
 as important spaces ranging from number theory, algebraic
 geometry, differential geometry, to topology etc.
 
 }
 \end{rem}

 \begin{prop} The space $\A_\infty$ is a complex space and 
 has a canonical stratification induced
 from the canonical stratification of the subspaces
 $\A_{g+1}-\A_g$, $g\geq 1$.
 \end{prop}
 
 Next we construct the completion $\overline{\A}_\infty^{SBB}$
  of $\A_\infty$.
 For this purpose, we follow the standard procedure of compactifications of arithmetic locally symmetric spaces
 in \cite{bb} (see also \cite{bj} for other references). 
 
 As mentioned before, the closure of $\h_g$ in its compact dual
 gives a compactification $\overline{\h_g}$,
 and the standard embedding of $\h_{g'}$, $g'< g$, into the
 boundary of $\overline{\h_g}$ in Equation \ref{embed-bdy}
 and the translates by $Sp(2g, \Z)$ of these standard
 boundary components give 
 all the {\em rational boundary components} of $\h_g$.
 The union of $\h_g$ with the rational boundary components
 gives a partial compactification $\overline{\h_g}_{,\Q}$ with 
 the Satake topology such that
 the quotient
 $Sp(2g, \Z)\backslash \overline{\h_g}_{,\Q}$ is 
 $\overline{\A}_g^{SBB}$.
 
 For our purpose, we need to show that these constructions 
 for $\A_g$ are
 compatible with respect to natural embedding between them when $g$ increases.
 
 \begin{prop}\label{compatible} For every $g$, 
 under the inclusion $\h_g\hookrightarrow \h_{g+1}$ 
 in Equation \ref{embed-h},
 the closure of $\h_g$ in $\overline{\h_{g+1}}_{,\Q}$ is equal to
 $\overline{\h_g}_{,\Q}$. Consequently,
 the closure of $\A_g$ in $\overline{\A}_{g+1}^{SBB}$ is equal to
 $\overline{\A}_{g}^{SBB}$.
 \end{prop}
 \begin{proof}
 This can be seen from how the standard boundary components
 $\h_{g'}$ of $\h_g$, $g'< g$, fit together and degenerate inductively.
 \end{proof}

 \begin{rem}\label{closure-a}
{\em  The compactification $\overline{\A}_{g+1}^{SBB}$
 can be written as a disjoint union
 
 $$\overline{\A}_{g+1}^{SBB}=\A_{g+1} \sqcup \A_g \sqcup \cdots
\sqcup \A_1 \sqcup (\A_0=\{\infty\}).$$
 Note that  $Sp(2, \Z)=SL(2, \Z)$
 and $\h_1$ is the Poincar\'e upper half plane $\H^2$,
 and hence $\A_1\cong SL(2, \Z)\backslash \H^2$ is noncompact
 and can be compactified by adding a cusp point $\{\infty \}$, which is really $\A_0$.  
 This shows that $\A_g$ is embedded into $\overline{\A}_{g+1}^{SBB}$ in two different ways: as an interior space
 through the embedding in Equation \ref{embed-h} and
 in the boundary through the above disjoint decomposition
 (or Equation \ref{embed-bdy}).
 }
 \end{rem}
 
 Once we have the compatibility in Proposition \ref{compatible},
 we can construct a completion $\overline{\A}_\infty^{SBB}$ as follows.
 
 From the increasing sequence of bordifications 
 $$\overline{\h_1}_{, \Q}\hookrightarrow \overline{\h_2}_{, \Q}
 \hookrightarrow \overline{\h_3}_{, \Q} \hookrightarrow \cdots,$$
 we can form 
 $$\overline{\h_\infty}_{, \Q}=\lim_{g\to \infty} \overline{\h_g}_{, \Q}.$$
 
 The space $\overline{\h_\infty}_{, \Q}$ also has a concrete realization similar to that of $\h_\infty$ in Equation \ref{h-infty}:
 
 \begin{equation}
 \overline{\h_\infty}_{, \Q}=\{
 \begin{pmatrix} X & 0\\ 0 & 0\end{pmatrix}
+ i \begin{pmatrix} Y & 0\\ 0 & I_\infty\end{pmatrix}
\mid X+iY \in \overline{\h_g}_{, \Q} \text{\ for some $g$}.\}
\end{equation}

 Taking the quotient by $Sp(\infty, \Z)$, we obtain
 the {\em desired completion} of $\A_\infty$,
 $$\overline{\A}_\infty^{SBB}=Sp(\infty, \Z) \backslash 
 \overline{\h_\infty}_{, \Q}.$$
 
 We note that 
 \begin{equation}
 \overline{\A}_\infty^{SBB}=\lim_{g\to \infty} \overline{\A}_g^{SBB}=\cup_{g\geq 1} \overline{\A}_g^{SBB}
 \end{equation}
  under the inclusion $\overline{\A}_g^{SBB} \hookrightarrow 
  \overline{\A}_{g+1}^{SBB}.$
  
  Motivated by the decomposition of $\overline{\A}_{g+1}^{SBB}$ in Remark \ref{closure-a}, we can obtain a decomposition of $\overline{\A}_\infty^{SBB}$ into an infinite dimensional
  interior and finite dimensional boundary pieces:
 
  \begin{prop}\label{a-appearance}
  The completion $\overline{\A}_\infty^{SBB}$ admits the following
  decomposition, 
  $$\overline{\A}_\infty^{SBB}=\A_\infty \coprod (\A_0 \sqcup \A_1 \sqcup \A_2 \sqcup \cdots),$$
  where the disjoint union 
  $$\sqcup_{g\geq 0} \A_g$$ 
  is the boundary, and $\A_\infty$ is the interior in some sense,
  which can also be decomposed into a non-disjoint union of $\A_g$, $g\geq 0$ (Equation \ref{a-defi}).
    Note  also that 
every $\A_g$ can appear in  
$\overline{\A}_\infty^{SBB}$ in two ways: either in the interior
 $\A_\infty$, or in the boundary $\cup_{g\geq 0} \A_g$.
 \end{prop}
 
 \section{Universal moduli spaces of Riemann surfaces}
 
 In this section, we construct a universal moduli space
 $\M_\infty$ and its completion $\overline{\M}_\infty^{SBB}$
 by using the spaces $\A_\infty$ and $\overline{\A}_\infty^{SBB}$
 constructed in the previous section.
 
 Recall that for every $g$, there is the embedding by the Jacobian
 map
 \begin{equation}\label{j1}
 J: \M_g\to \A_g,
 \end{equation}
 which induces an embedding
 \begin{equation}\label{j2}
 J: \overline{\M}_g^{SBB} \to \overline{\A}_g^{SBB}.
 \end{equation}
 
 We note the following description of the boundary of $\overline{\M}_g^{SBB}$.
 \begin{prop}\label{mg-stratification}
 The boundary of $\overline{\M}_g^{SBB}$
 is the union of $\M_{g_1}\times \cdots\times \M_{g_k}$,
 where $g_1+\cdots + g_k\leq g$, $k\geq 1$.
 The equality occurs only when we pinch homologically
 trivial loops of compact Riemann surfaces $\Sigma_g$ of
 genus $g$.
 \end{prop}
 
 To see this, we note that if we pinch a homologically nontrivial loop
 on $\Sigma_g$, then we get a surface $\Sigma_{g-1}$
 on the boundary of $\overline{\M}_g^{SBB}$. 
 If we pinch one homologically trivial loop, then we
 get a union of two compact Riemann surfaces
 $\Sigma_{g_1}$, $\Sigma_{g_2}$, where $g_1+g_2=g$.
 By iteration, we get the above picture. (Note that if we use the Deligne-Mumford compactification of $\M_g$, then we get punctured Riemann surfaces, and the locations of the punctures
 make the dimension of the boundary components bigger.
  Here we forget these punctures).
 
 \begin{prop}
 For every pair of natural numbers $g' < g$,  if the moduli space $\M_{g'}$
 appears in the boundary of $\overline{\M}_g^{SBB}$ as in the 
 above proposition, then the closure of $\M_{g'}$
 is equal to $\overline{\M}_{g'}^{SBB}$. 
 \end{prop}

This is one nice inductive property of the compactification
$\overline{\M}_g^{SBB}$. No new types of Riemann surfaces and 
their moduli spaces appear.

For the purpose of constructing a universal
moduli space of Riemann surfaces, a seemingly unfortunate fact
is that there is no obvious inclusion of $\M_g$ into $\M_{g+1}$
as in the case of $\A_g\subset \A_{g+1}$.
On the other hand, the Jacobian maps $J$ in Equations \ref{j1}, \ref{j2} overcome this difficulty. Later, this turns out to be a nicer
property in terms of constructing stratifications. 

Consider the subspaces $J(\M_g)$ of $\A_g\subset \A_\infty$ and 
$J(\overline{M}_g^{SBB})$ of $\overline{\A}_g^{SBB}
\subset \overline{\A}_\infty^{SBB}$, $g\geq 1$.
Define
$$\M_\infty=\cup_{g\geq 1} J(\M_g) \subset \A_\infty,$$
and 
$$\overline{\M}_\infty^{SBB}=\cup_{g\geq 1} J(\overline{\M}_g^{SBB})
\subset \overline{\A}_\infty^{SBB}.$$

After defining these spaces, it is crucial to understand their properties. The next result is probably the most basic or the minimal requirement, otherwise
we could take the trivial construction of a disjoint union of $\overline{\M}_g^{SBB}$,
which is definitely not what we want.

\begin{prop}
The subspace $\overline{\M}_\infty^{SBB} \subset \overline{\A}_\infty^{SBB}$ is connected.
\end{prop}
\begin{proof}
We want to show that $J(\M_g)\subset \A_g$ is 
contained in the closure of $J(\M_{g+1})$ in $\A_{g+1}$.
Recall that $\A_g$ is embedded into $\A_{g+1}$ through the
embedding of $\h_g\hookrightarrow \h_{g+1}$ in Equation
\ref{embed-h}.
Suppose $\Sigma_1$ is the compact Riemann surface of genus $1$ whose period in $\h_1$ is equal to $i$ with respect to a suitable
choice of basis of $H_1(\Sigma_1, \Z)$.
For any compact Riemann surface $\Sigma_g$ of genus $g$,
$J(\Sigma_g)$ gives a point $p$ in $J(\M_g)\subset \A_g$.
Let $p$ also denote the image of $p$ in $\A_{g+1}$
under the above embedding $\A_g \to \A_{g+1}$.
Then the disjoint union $\Sigma_g \cup \Sigma_1$ is 
mapped to the point $p\in \A_{g+1}$.

Now if we pick two points on $\Sigma_g$ and $\Sigma_1$ 
and remove small disks around them depending on a small parameter
$\varepsilon$
  and glue them, we get a
compact Riemann surface $\Sigma_{g+1, \varepsilon}$
of genus $g+1$ with a short separating neck. 
Note that $J(\Sigma_{g+1, \varepsilon})$ is contained
in $J(\M_{g+1})$ in $\A_{g+1}$. When $\varepsilon\to 0$, $J(\Sigma_{g+1, \varepsilon})$ converges to
$J(\Sigma_g \cup \Sigma_1)$, which is the point $p$
in $\A_{g+1}$ above. 
It follows that $p$ is in the closure of $J(\M_{g+1})$,
and hence 
 every point of $J(\M_g)$ is a limit
of points of $J(\M_{g+1})$.

\end{proof}

\begin{rem}
{\em Adding a compact Riemann surface $\Sigma_1$  of
genus 1 to $\Sigma_g$ to obtain a compact Riemann surface
$\Sigma_{g+1, \varepsilon}$ is one natural way to relate
$\M_g$ to $\M_{g+1}$.
This was used in the formulation of stability results in 
\cite{har} and \cite{mw} on homology and cohomology of
$\M_g$, or mapping class groups.
}
\end{rem}

\begin{prop}
The subspace $\overline{\M}_\infty^{SBB} \subset \overline{\A}_\infty^{SBB}$  has a canonical stratification such that the closure
of each stratum is a projective variety over $\C$,
and is a union of $\overline{\M}_g^{SBB}$,  $g\geq 0$, 
though $\overline{\M}_g^{SBB}$ can appear in many different ways
in $\overline{\M}_\infty^{SBB}$.
\end{prop}
 \begin{proof}
 We note that each $\overline{\M}_g^{SBB}$ has a canonical stratification by Proposition \ref{mg-stratification}. 
 Since $\overline{\M}_\infty^{SBB}$ is the union of
 $\overline{\M}_g^{SBB}=J(\overline{\M}_g^{SBB})$  for $g \geq 0$, 
 the above Proposition shows that $\overline{\M}_g^{SBB}$
 is contained in $\overline{\M}_{g+1}^{SBB}$.
 Then by   considering
  $\overline{\M}_{g+1}^{SBB} -\overline{\M}_{g}^{SBB}$,
  we obtain a desired stratification.
 \end{proof}
 
 In comparison to Proposition \ref{a-appearance} about the decomposition of $\overline{\A}_\infty^{SBB}$, we have the following
  result.
  
  \begin{thm}
  For every $g$, there is only one way to embed
  $\M_g$ into $\overline{\M}_{g+1}^{SBB}$, which is also equal
  to the closure of $\M_{g+1}$ inside $\overline{\M}_\infty^{SBB}$.
  Under this inclusion, we get an increasing sequence of spaces:
  $$\overline{\M}_1^{SBB}\hookrightarrow \overline{\M}_2^{SBB}\hookrightarrow  \overline{\M}_3^{SBB}\hookrightarrow  
  \cdots,$$
  and 
  $$ \overline{\M}_\infty^{SBB}=\lim_{g\to \infty} \overline{\M}_g^{SBB}=\cup_{g\geq 1} \overline{\M}_g^{SBB}.$$
  
  \end{thm}
  
 \begin{rem}
 {\em 
  Note that the unique embedding of $\M_g$ into the compactification $\overline{\M}_{g+1}^{SBB}$ 
  is in some sense a nicer property for this family of $\M_g$ than
  the family of $\A_g$ since there are two different
  embeddings of $\A_g$ into $\overline{\A}_{g+1}^{SBB}$,
  as pointed out in Proposition \ref{a-appearance}.
  }
  \end{rem}
 
 {\begin{rem}
{\em The relation between the inductive limit $\lim_{g\to \infty} \overline{\M}_g^{SBB}$ from the general construction
 and the space constructed in this paper  through 
$\overline{\A}_\infty^{SBB}$ can also be seen as follows. A Riemann
surface of genus $g$ on one hand is a degenerated Riemann surface $\Sigma$ of genus $g+1$ where a homologically nontrivial loop, i.e., a non-separating loop,  has been pinched to a point. Therefore, its Jacobian $J(\Sigma)$, an Abelian variety of dimension $g$,  is also a degenerated Abelian variety of dimension $g+1$. Alternatively, we can identify $J(\Sigma)$ with an Abelian variety of dimension $g+1$ by multiplying it with a normalized Abelian variety of dimension 1. This would correspond to viewing $\Sigma$ as a Riemann surface of genus $g+1$ by taking its disjoint union with a standard Riemann surface of genus 1. Of course, there is the issue of the choice of normalization here for that Riemann surface of genus 1. But the advantage of the construction is that we no longer need to go to the boundary of the moduli space for Riemann surfaces of genus $g+1$ or of principally polarized Abelian varieties of dimension $g+1$ to get the objects of genus/dimension $g$, but can stay within the interior. And we can interpolate between the two construction by degenerating the Riemann surfaces of genus 1 
that had been added as a factor/component. (We don't need to address the other way a Riemann surface of genus $g+1$ can be degenerate, by pinching a homologically trivial loop, because in that case, the genus and the dimension of the Jacobian do not drop, and therefore, the corresponding degeneration stays in the interior of the moduli space $\A_g$ anyway.)
}
\end{rem}}

\section{Riemannian metrics}\label{sec:met}
First, we note that each irreducible symmetric space has a
unique invariant Riemannian metric up to scaling.
On each $\h_g$, we can choose  the invariant  Riemannian metric such that under the canonical embedding 
$\H^2=\h_1\hookrightarrow \h_g$ as in (or induced inductively from) Equation \ref{embed-h}, the induced metric
on $\H^2$ is the Poincar\'e hyperbolic metric.

\begin{prop}
With the above normalization of invariant metrics on $\h_g$,
 for every $g\geq 1$,  the  embedding
$\h_g \hookrightarrow \h_{g+1}$ in Equation \ref{embed-h} is an isometric embedding,
and the infinite dimensional Siegel space $\h_\infty$ has an invariant Riemannian metric,
which induces the normalized  invariant Riemannian metric on each embedded interior subspace $\h_g$ in $\h_\infty$.
\end{prop}

On the completion $\overline{\h_\infty}_{, \Q}$, we can put
on  a stratified Riemannian metric so that on each standard boundary component
$\h_g$ and hence every rational boundary component,
the induced metric is the above normalized invariant
metric on $\h_g$.

Though the boundary strata $\h_g$ are at infinite distance from interior points of $\overline{\h_\infty}_{, \Q}$,
i.e., points in $\h_\infty$ (for example, from the interior points 
contained in any interior subspace of $\h_{g'}$ of $\h_\infty$), 
 it is no problem
since these metrics on the boundary strata are compatible 
in tangential directions in the following sense.

Suppose $\h_g$ is a rational boundary component, i.e., 
contained in the boundary of $\overline{\h_\infty}_{, \Q}$. Then we have families of ``parallel"  subspaces
$\h_g$ inside $\h_\infty$ which converge to the boundary component $\h_g \subset \overline{\h_\infty}_{, \Q}$.
For example, for the standard boundary component $\h_g$
of $\overline{\h_\infty}_{, \Q}$, we can push the canonically embedded interior subspaces 
$ \h_g$ in $\h_\infty$  towards the boundary component.

If $v $ is a tangent vector to such an interior subspace $\h_g$, 
then it is also
a tangent vector to the boundary $\h_g$. 
An important point  is that the norms of $v$ are the same.

This means that we have a compatible stratified Riemannian metric on different stratification components of 
the completion $\overline{\h_\infty}_{, \Q}$, and hence also
on the completion $\overline{\A}_\infty^{SBB}$.

Now when we decompose $\overline{\A}_\infty^{SBB}$ into
the disjoint union 
\begin{equation}
\overline{\A}_\infty^{SBB}= \A_\infty\coprod \A_0 \sqcup \A_1 \sqcup \A_2 \cup \cdots \cup \A_g
\sqcup \cdots,
\end{equation}
we can use the Riemannian metric to define a measure on each of the boundary piece $\A_g$.
(Note that they are disjoint).

For the interior $\A_\infty$,  which is a {\em non-disjoint union}
 of $\A_0,  \cdots,  \A_g, \cdots$,
and for any finite dimensional analytic subspace $K$ in 
$\A_\infty$, we can use
the Riemannian metric of $\A_\infty$ or $\h_\infty$ to define a measure on $K$, or suppose $K$ is contained in some $\A_g$,
then we can use Riemannian measure of $\A_g$ and restrict
it to $K$.

The double appearance of $\A_g$ in $\overline{\A}_\infty^{SBB}$ above
in  might not be so nice.
On the other hand, this does not occur for $\overline{\M}_\infty^{SBB}$.

\vspace{.1in}
\noindent{\bf Stratified Riemannian metric on $\overline{\M}_\infty^{SBB}$} 
\vspace{.1in}

To construct a measure on $\overline{\M}_\infty^{SBB}$, we use the embedding
$$\overline{\M}_\infty^{SBB}  \subset \overline{\A}_\infty^{SBB}.$$

We can pull back the stratified Riemann metric on $\overline{\A}_\infty^{SBB}$ to $\overline{\M}_\infty^{SBB}$. Since, however, this is not an immersion, the pull back of the stratified Riemann metric of $\overline{A_\infty}$ is not everywhere positive definite on  $\overline{M_\infty}$ \cite{r1,r2}. We will address that issue in a moment and first investigate the properties of this pull-back metric.

To describe this metric, we use the following {\em disjoint decomposition}:

\begin{equation}
\overline{\M}_\infty^{SBB}=
 \sqcup_{g\geq 1} \M_g  \coprod_{k\geq 2,
g_1+ \cdots + g_k=g} \M_{g_1} \times \cdots \times \M_{g_k}.
\end{equation}

Note that in the above decomposition, we remove the distinguished boundary component $\M_{g-1}$ of  
$\overline{\M_g}^{SBB}$ and its boundary components in Proposition \ref{mg-stratification}
in order to avoid repeatation, and group other components together
with $\M_g$.

This distinguished boundary component $\M_{g-1}$
of $\overline{\M_g}^{SBB}$ is also  at
infinite distance from the interior points of $\M_g$ and hence of
$\overline{\M_g}^{SBB}$. (There are other boundary components
at infinite distance from the interior points which result from 
pinching homologically nontrivial loops. They appear in the boundary of 
$\M_{g-1}$).
On the other hand, the Riemannian metric on the boundary
component $\M_{g-1}$ and the Riemannian metric on the interior
of $\overline{\M_g}^{SBB}$ 
are compatible in a similar sense as described above, i.e.,
when interior points  of $\overline{\M_g}^{SBB}$ converge to
a point in $\M_{g-1}$, the norms of tangential vectors of $\M_{g-1}$ converge.
Therefore,  we can take the corresponding measures on all these strata, in particular, $\M_g$, to get a 
compatible stratified measure on $\overline{\M_g}^{SBB}$.

We now address the issue of the non-positive definiteness, following \cite{hj1}. The  Jacobian map $J$ that we have used for the embedding $\M_g \to \A_g$ 
associates to each marked Riemann surface  $\Sigma$ its Jacobian $J(\Sigma)$, a principally
polarized Abelian variety. In particular, $J(\Sigma)$
has  canonical flat metric.

Instead of
using the  Jacobian  map in order to map $\M_g$ to $\A_g$, we may use a related period map, which is a map 
from each individual Riemann surface $\Sigma$ to $J(\Sigma)$, 
$$p_\Sigma:  \Sigma \to J(\Sigma),$$
in order to obtain a metric on
$\Sigma$ by pulling back the flat metric of $J(\Sigma)$.
This metric on $\Sigma$ which is called the {\em Bergman metric} can
also be described as follows. Let $\theta_1,\dots ,\theta_g$ be an
$L^2$-orthonormal basis of the space of holomorphic 1-forms on $\Sigma$, i.e.
\bel{met1}{\sqrt{-1}\over 2}
  \int_\Sigma\theta_i\land\bar\theta_j=\delta_{ij}.
\qe
Note that since the sum of the squares of an orthonormal basis of holomorphic one-forms (orthonormality is conformally invariant for one-forms), \rf{met1} does not depend on the prior choice of a metric on $\Sigma$.

The Bergman metric then is simply given by
$$\rho_B(z)dzd\bar z:=\sum_{i=1}^g\theta_i\bar\theta _i.$$
We can form the $L^2$-product of holomorphic quadratic differentials
\bel{met3}
(\omega_1 dz^2,\omega_2dz^2)_{B}:=
  {\sqrt{-1}\over2}\int_\Sigma\omega_1(z)\overline{\omega_2(z)}
  {1\over\rho_B(z)}dz\land d\bar z.
\qe
This induces a  Riemannian metric on the Teichm\"uller space $\mathcal T_g$ and therefore also on its quotient $\M_g$ (or, more precisely, on a finite cover of $\M_g$  that does not possess quotient singularities). (Note that $\mathcal T_g$ is a complex
manifold  and is the universal covering of $\M_g$ as an orbifold.)
 For simplicity, we shall also call this metric on $\M_g$ the Bergman metric. As shown in \cite{hj1}, this metric dominates the Siegel metric just described, i.e., the metric using the
 map $J: \M_g \to \A_g$. In order to understand this result, one should note that both metrics are induced by the 
Jacobian map that associates to each Riemann surface its Jacobian. For the metric $(.,.)_B$, we use the period map on each  individual Riemann surface $\Sigma$, and pull back the flat metric on its Jacobian $J(\Sigma)$, and then form an $L^2$-product, analogous to the construction of the Weil-Petersson metric that works with the hyperbolic instead of the Bergman metric on $\Sigma$. For the Siegel metric, in contrast, we do not pull back the flat metrics on the individual Jacobians, but rather the symmetric metric on the moduli space $\A_g$ of principally polarized Abelian varieties. In contrast to the Siegel metric, $(.,.)_B$ is positive definite everywhere. It has the same asymptotic behavior as the former, however. We summarize our considerations in the following
\begin{thm} There exists a  metric on  $\overline{\M_g}^{SBB}$ that induces a Riemannian metric on each stratum $\M_g$. This metric has the following properties
\begin{itemize}
\item Boundary points of $\M_g$ corresponding to pinching a homologically
nontrivial loop are at infinite distance from the interior.
\item
Boundary points of $\M_g$ corresponding to pinching a homologically
(but not homotopically) trivial loop have finite distance from the
interior.
\item
The metric forgets the punctures, i.e.\  the boundary components
have codimension 3, again with the exception of the one where
one component of the limit Riemann surface has genus 1.
\end{itemize}
\end{thm}

\begin{cor}
There exists a measure on $\overline{\M_g}^{SBB}$ that induces a finite volume measure on each stratum $\M_g$. 
\end{cor}
\proof
Both the Siegel metric and the Bergman metric induce measures on the strata of 
$\overline{\M_g}^{SBB}$. Even though the Siegel metric is not positive definite, it is degenerate only on the hyperelliptic locus because that is where the Jacobian map $J: \M_g\to \A_g$
 is not an immersion as orbifolds, but this hyperelliptic locus is a quasiprojective subvariety of lower dimension. Thus, we find not only one, but two measures satisfying the claim. The measure induced by the Bergmann dominates that induced by the Siegel metric. 
\endproof \qed

Thus, for the purposes of integration theory in Section \ref{sec:int},  we could use either one. In order to
get convergence of integrals on $\overline{\M}_\infty^{SBB}$, we shall have to choose  weights
for these  measures on  the components $\M_g$ depending on $g$. 
\begin{rem}
{\em
When we shall combine measures on subspaces of a common ambient
space to define a global measure in Section \ref{sec:int} below, one issue is the compatibility.
The above explanation shows that the canonical Riemannian
metric on $\h_\infty$ and $\overline{\h_\infty}_{, \Q}$ serves as a gauge to coordinate different components $\M_g$. 
When the dimension of the strata jumps, it does not affect
distance functions much,  but it has a big impact on measures
by noticing that a subspace of smaller dimension usually 
has zero measure with respect to an absolutely continuous measure on an ambient space such as Riemannian measures. 
Therefore, measures on different subapaces
need to be  adjusted according to their dimensions. This is what we shall now turn to.
 }
 \end{rem}

{\section{Integration on the universal moduli space of Riemann surfaces}\label{sec:int}

The considerations in this section will apply to both the inductive limit of the $\overline{\M}_g^{SBB}$ for $g\to \infty$ and the space $\overline{\M}_\infty\subset \overline{\A}_\infty^{SBB}$ constructed in this paper. They will not directly apply to $\overline{\A}_\infty^{SBB}$, because that space is not a disjoint union of strata corresponding to the different values of $g$.

We want to construct a measure on our space with respect to which every stratum has finite volume. This measure will be inductively built from measures on the spaces $\overline{\M}_g^{SBB}$. First of all, we note that $\A_g$ is a finite volume quotient of the Siegel upper half space with its Riemann-Lebesgue measure induced by its natural Riemannian metric induced by the symmetric structure. This then induces a measure $\mu_g$ on $\M_g$. That latter measure also has finite volume, essentially for the same reason that the Poincar\'e metric on the punctured disk has locally finite volume near the puncture.}

{The subsequent constructions will work for both the Siegel and the Bergman metric as explained in Section \ref{sec:met}.}

{We should recall  that the different boundary components  of $\M_g$ behave differently with respect to our metric, be it the Siegel or the Bergman metric. 
The metric on $\M_g$ induced from the invariant metric on $\A_g$
under the Jacobian map $J$ is not of Poincar\'e type
near a boundary component $\M_i\times \M_{g-i}$, since the induced metric here is not complete near such boundary points.
For the boundary component $\M_{g-1}$ lying on the boundary of
$\A_g$, it is complete. Nevertheless, our construction of measures will work in either situation, although the behavior of the measure will be different according to the type of boundary component.}

{ We can then build the measure
\be
\mu=\sum_g \lambda_g \mu_g
\qe
with positive real numbers $\lambda_g$ to be chosen. This means that
\be
\mu(A)=\sum_g \lambda_g \mu_g(A\cap \M_g)
\qe
for every measurable subset $A$ of our space, where again measurability requires measurability of the intersection with every stratum.}

{Again, we should point out that 
the space $\overline{\M}_\infty^{SBB}$ is not simply a disjoint union of the different $\M_g$.
The stratification is a bit more complicated. For example,
$\M_i\times \M_{g-i}$, $1\leq i \leq g-1$,  and products of more factors appear also
in the boundary of $\overline{M}_g^{SBB}$. Nevertheless, this does not affect our construction. 
}

{ (The principle of the construction can easily be seen by considering the closed unit interval $[0,1]$ and equip it with the $\lambda_1$ times the Lebesgue measure on the open interval $(0,1)$ plus $\lambda_0$ times the Dirac measures at the boundary points $0$ and $1$.)

The question then arises how to determine the $\lambda_g$. For the convergence of series as they occur in string theory, the $\lambda_g$ should be sufficiently rapidly decaying functions of $g$. In string theory, this is achieved as follows. Let $\Sigma$ be a Riemann surface of genus $g$, and $h:\Sigma\to \R^N$ be a Sobolev function. Then its Dirichlet integral \rf{intro1}, also called the Polyakov action (see e.g. \cite{pol,j2}) in string theory, is
\bel{int3} 
S(h,\Sigma)=\int_\Sigma |dh|^2
\qe
where by conformal invariance, we do not need to specify a conformal metric on $\Sigma$, as explained above. The string action then is defined as
\be
S_{string}(h,\Sigma):= S(h,\Sigma) + \alpha (2g-2)
\qe
for some positive constant $\alpha$ that needs to be determined by considerations from physics which are not relevant for our current purposes (see \cite{dhp,pol}.

\begin{thm}
  For $N=26$,  the string partition function can be written as
\be
\sum_{g\ge 0} \int \exp(S_{string}(h,\Sigma)dh d\Sigma
\qe
with a functional integration which for the component $\Sigma$ is carried out over $\M_g$. 
\end{thm}
\proof
For $N=26$,  the conformal anomalies cancel (see e.g. \cite{pol,j2}), and consequently, for each genus $g$, the partition function can be written as an integral over the moduli space of Riemann surfaces of genus $g$  and an integral over the field $h$, and for each $g$, the value of this integral is finite. The term $\alpha (2g-2)$ in \rf{int3} then ensures the convergence of the resulting series. \endproof

As we have just explained , measures can be
constructed on both $\lim_{g\to \infty} \overline{\M_g}^{SBB}$
via the general construction of direct limit and as a subspace
of the infinite dimensional locally symmetric space
$\A_\infty$ and its completion $\overline{\A}_\infty^{SBB}$.

We shall now explain that the additional structure coming from $\overline{\A}_\infty^{SBB}$ and the embedding
$J: \overline{\M}_\infty^{SBB}\to \overline{\A}_\infty^{SBB}$
will shed some light on how natural the above construction is.

First, we point out that infinite dimensional symmetric spaces
based on Hilbert spaces have natural Riemannian metrics,
and their submanifolds also have induced Riemannian metrics.
Unlike the case of finite dimensional cases,
Riemannian metrics on infinite dimensional manifolds do not automatically induce measure
and integration theory. Integration on infinite dimensional
manifolds based on Hilbert and Banach spaces seems  complicated. See  \cite{ku} \cite{wei}, for example, for more information and references.

On the other hand, spaces in this paper such as $\h_\infty$,
$\overline{\A}_\infty^{SBB}$ and $\overline{\M}_\infty^{SBB}$
have filtrations and stratifications by finite dimensional 
submanifolds, and we are only interested in finite dimensional
subspaces at each step in some sense,
and the general idea of analysis on stratified
spaces will help. One important point is to use
an invariant metric on the infinite dimensional
symmetric space $\h_\infty$ and its completion
$\overline{\h_\infty}_{, \Q}$ to coordinate metrics (see the
previous section)
and hence 
measures on these different strata.
See the book \cite{pf} for some related information and references. 

This is precisely what we have achieved here, as we have constructed  a measure on $\overline{\A}_\infty^{SBB}$ and $\overline{\M}_\infty^{SBB}$ by using a stratified 
Riemann metric on $\overline{\h_\infty}_{, \Q}$.

\section{Infinite dimensional locally symmetric spaces and stable cohomology of arithmetic groups}

The Siegel upper half plane $\h_g$ is one important generalization
of the Poincar\'e upper half plane $\H^2=SL(2, \R)/SO(2)$.

Another important generalization is 
$X_n=SL(\infty, \R)/SO(n)$, $n\geq 2$.
The arithmetic subgroup $SL(n, \Z)$ acts properly on $X_n$,
and the quotient $SL(n, \Z)\backslash X_n$ is the moduli space of
positive definite $n\times n$-matrices 
(or quadratic forms in $n$ variables) of determinant 1,
and which is also 
 the moduli space of flat tori of volume 1 in dimension $n$.

Clearly there is an embedding 

$$X_n\hookrightarrow X_{n+1}, \quad A\mapsto 
\begin{pmatrix} A & 0 \\ 0 & 1\end{pmatrix},$$
and an embedding
$$ SL(n, \R) \hookrightarrow SL(n+1, \R), \quad
\quad A\mapsto 
\begin{pmatrix} A & 0 \\ 0 & 1\end{pmatrix}.$$

Then the procedure in \S 2 goes through,
and we can construct
the infinite dimensional symmetric space 
$X_\infty=\lim_{n\to \infty} X_n$, the infinite dimensional
Lie group $SL(\infty, \R)=\lim_{n\to \infty} SL(n, \R)$,
and its arithmetic subgroup $SL(\infty, \Z)= \lim_{n\to \infty} SL(n, \Z)$.

These spaces and groups can be realized concretely as follows:
$$X_\infty=\{ \begin{pmatrix} A & 0\\ 0 & I_\infty\end{pmatrix}
 \mid A
\text{\ is a positive definite $n\times n$-matrix for some $n\geq 1$}, \det A=1 \},$$

$$SL(\infty, \R)=\begin{pmatrix} A & 0\\ 0 & I_\infty\end{pmatrix}
 \mid A \in SL(n, \R) 
\text{\  for some $n\geq 1$} \},$$

and 
$$SL(\infty, \Z)=\begin{pmatrix} A & 0\\ 0 & I_\infty\end{pmatrix}
 \mid A \in SL(n, \Z) 
\text{\  for some $n\geq 1$} \}.$$

The quotient $$SL(\infty, \Z)\backslash X_\infty$$ is an infinite
dimensional locally symmetric space which contains
every $SL(n, \Z)\backslash X_n$.

It is known that each $SL(n, \Z)\backslash X_n$ has a minimal
Satake compactification
$$\overline{SL(n, \Z)\backslash X_n}^S=
SL(n, \Z)\backslash X_n\sqcup SL(n-1, \Z)\backslash X_{n-1}
\sqcup \cdots \sqcup \{ \infty \},$$
which can be obtained from the quotient by an $SL(n, \Z)$-action
on a completion $\overline{X_n}_{,\Q}$ by adding rational boundary components of $X_n$.

For every $n$, the embedding $X_n\hookrightarrow X_{n+1}$ induces an embedding
$$\overline{SL(n, \Z)\backslash X_n}^S \hookrightarrow 
\overline{SL(n+1, \Z)\backslash X_{n+1}}^S.$$

Similarly, we can construct a completion $\overline{X_\infty}_{,\Q}$
of $X_\infty$
which is invariant under $SL(\infty, \Z)$ and hence
a completion $\overline{SL(\infty, \Z)\backslash X_\infty}^S$
of $SL(\infty, \Z)\backslash X_\infty$
such that for every $n$,
$$\overline{SL(n, \Z)\backslash X_n}^S
\hookrightarrow \overline{SL(\infty, \Z)\backslash X_\infty}^S,$$
and
$$\overline{SL(\infty, \Z)\backslash X_\infty}^S=\cup_{n\geq 1}
\overline{SL(n, \Z)\backslash X_n}^S=\lim_{n\to\infty}
\overline{SL(n, \Z)\backslash X_n}^S.$$

It is reasonable to expect that this same construction works for
other series of classical simple algebraic groups.

These spaces should be related to the stability of cohomology
of arithmetic groups in \cite{bo}. In fact, the analogy of the inclusion
$X_n\hookrightarrow X_{n+1}$ was used in \cite{bo} to
formulate and prove the stability result there.

\section{Some other reasons for constructing universal
moduli spaces}

The stability results for homology groups of the mapping class groups of surfaces, or equivalently, the homology groups
of $\M_{g,n}$, motivated Mumford conjecture
on stable cohomology of $\M_{g,n}$ \cite{mum} \cite{mw}.
The proof depends on realizing them as the cohomology
groups of  universal classifying  spaces. See \cite{wa} for a summary. These classifying spaces are topological objects
and well-defined up to homotopy equivalence.

We note that $\M_g$ is a classifying space for the mapping
class group of a compact surface $S_g$ of genus $g$
 for rational coefficients. One natural question is to {\em construct
 a universal moduli space which is an infinite dimensional
 algebraic variety and enjoys  some good
 algebraic geometry properties  and whose cohomology groups realize
 the stable cohomology groups as conjectured by Mumford.}
 We don't  know if our space $\M_\infty$ and its completion
 $\overline{\M}_\infty^{SBB}$ might be helpful to this purpose.

 We also note that for classical families of compact Lie groups,
 their limits and limits of their classifying spaces are univresal
 groups and universal spaces, and they 
 are important in characteristic classes. See \cite[\S 5]{mi}.

In the famous 
Bott periodicity theorem \cite{bot}, 
limits of increasing sequences of classical compact Lie groups
and associated spaces also appear naturally.

All these results explain that, beisdes the applications to
string theory and the theory of minimal surfaces,
  it is a natural and important problem to
consider universal moduli spaces of Riemann surfaces and
universal (or rather infinite dimensional) symmetric spaces of noncompact type and
their quotients by arithmetic groups.

\vspace{.2in}
\noindent{\em Acknowledgments}: We would like to thank
K.-H. Neeb for helpful conversations on infinite
dimensional symmetric spaces and for pointing out several
references. \\
This work was partially supported by a Simons Fellowship from the Simons Foundation 
(grant no.305526) to the first author and an ERC Advanced Grant (FP7/2007-2013 / ERC grant agreement n$^\circ$~267087)  to the second author.

\end{document}